\documentclass{article}

\usepackage{arxiv}

\usepackage{etoolbox}
\robustify\bfseries%
\usepackage{lmodern}
\usepackage[final]{microtype}
\usepackage[utf8]{inputenc}
\usepackage{amsmath, array, mathtools, amssymb, amsthm}
\usepackage{hyperref}
\usepackage[capitalize]{cleveref}
\usepackage[ruled]{algorithm2e}
\usepackage[binary-units]{siunitx}
\usepackage{booktabs}
\usepackage{subcaption}
\usepackage{phd}
\usepackage{xcolor}
\usepackage[round]{natbib}

\usepackage{caption}
\newtheorem{proposition}{Proposition}

\renewcommand{\headeright}{Preprint}
\renewcommand{\undertitle}{Preprint}

\title{Faster SVM Training via Conjugate SMO}

\author{Alberto Torres-Barr\'{a}n \\
        Instituto de Ciencias Matemáticas (ICMAT) \\
        Centro Superior de Investigaciones Científicas (CSIC) \\
        Nicolás Cabrera, nº13-15 \\
        28049 Madrid, Spain \\
        \texttt{alberto.torres@icmat.es} \And 
        Carlos M. Ala\'{i}z \\
        Departamento de Ingeniería Informática \\
        Universidad Autónoma de Madrid (UAM) \\
        Francisco Tomás y Valiente, 11 \\
        28049 Madrid, Spain \\
        \texttt{carlos.alaiz@uam.es} \And 
        Jos\'{e} R. Dorronsoro \\
        Departamento de Ingeniería Informática and Instituto de Ingeniería del Conocimiento \\
        Universidad Autónoma de Madrid (IIC-UAM)\\
        Francisco Tomás y Valiente, 11 \\
        28049 Madrid, Spain \\
        \texttt{jose.dorronsoro@uam.es}}

\begin{document}
\maketitle

\begin{abstract}%
We propose an improved version of the SMO algorithm for training classification and regression SVMs, based on a Conjugate Descent procedure. This new approach only involves a modest increase on the computational cost of each iteration but, in turn, usually results in a substantial decrease in the number of iterations required to converge to a given precision.
Besides, we prove convergence of the iterates of this new Conjugate SMO as well as a linear rate when the kernel matrix is positive definite.
We have implemented Conjugate SMO within the LIBSVM library and show experimentally that it is faster for many hyper-parameter configurations, being often a better option than second order SMO when performing a grid-search for SVM tuning.
\end{abstract}



\section{Introduction}
\label{sec:intro}
Support Vector Machines (SVMs) \citep{cortes1995support} received an enormous attention in the 1990's, not only because of the elegant optimization and risk minimization theories underlying them but also because kernel SVMs provided very powerful classification and regression models that often beat the classical neural networks at that time.
However, at least in their customary Gaussian kernel formulation, they seemed to have currently lost some of their luster, partly because 
their traininig and prediction costs may be too high for the big data problems currently dominating Machine Learning (ML).
This is so because the number of Support Vectors (SVs) underlying any SVM model is usually linear with respect to the sample size.
This implies, first, that for a size $N$ sample, Sequential  Minimum Optimization (SMO), the standard training procedure to solve the dual problem of kernel SVMs, requires $\Omega(N)$ iterations each with $O(N)$ cost \citep{fan2005working}. As a consequence the training cost is, at least, $\Omega (N^2)$ (usually higher), even without counting kernel operations.
Besides, predicting a single new pattern will also have a $O(N)$ cost and, thus, for big sample size problems prediction costs may be too high.
This is the case in many big volume and/or velocity problems of today's big data.

But, on the other hand, big data is, at the end, a moving category, defined by problem size but also by the hardware available at a given time.
In fact advances in hardware imply that problems that 5-10 years ago would be considered big data, are not perceived as such now, since
nowadays computing nodes with RAM sizes of up to 1 TB and above 50 cores are relatively common in research environments.
This means that quite large caches can be used in such machines and, in turn, that sample sizes about $10^5$ patterns, even with relatively large dimensions, can also be dealt with. In addition, model hyper-parameterization can be largely sped up by simple core parallelization.
A well-known such example would be the MNIST problem, often used as a benchmark: it is currently far from the big data league but it is also representative of problems easily solved today but much less so 10 years ago.

Gaussian SVMs are also much more robust than other models when facing feature collinearity, and they are often hard to beat when the number of features is not very high \citep{Rudin_secrets}. The standard algorithm to solve non-linear SVMs, SMO \citep{fan2005working}, is very elegant and powerful, 
with simple, analytic steps and asymptotically linear convergence when the kernel matrix is positive definite \citep{chen2006study}.

Of course, there is a large on-going effort to adapt SVMs to big data settings.
For instance, an important component of the usefulness of Gaussian SVMs is the nonlinear projection of the original patterns in a new space  with a much larger dimension.
However, when the original pattern dimension is large enough, kernels may not be needed and one can work with linear SVMs, either solving the primal or dual problems; see for instance \cite{yuan2012improved}.
Training becomes much faster and large samples can be more easily handled than in the kernel case, but the bias term $b$ has to be dropped in order to get rid of the equality constraint it imposes on the dual problem.
In principle, dropping the offset should hamper the performance of a kernel SVM model, but this may not be always the case \citep{Steinwart_2011}.
On the other hand, unless pattern dimension is substantially high (at least in the thousands), the performance of Gaussian SVMs is usually better than the linear ones.
Here we shall consider Gaussian SVMs for both classification and regression problems, retaining the offset.

In any case, there is a huge literature on speeding up SVM training. 
For instance, early attempts to provide SVMs with online training are the well-known NORMA \citep{Kivinen} and Pegasos \citep{shalev2007pegasos} procedures.
Other recent proposals include decomposing large datasets in appropriate chunks \citep{Thomann}, applying multilevel techniques \citep{Schlag}, accelerating kernel operations using GPUs \citep{Ma}, approximating kernel operations using feature randomization \citep{Rahimi}, applying budget constraints to dual training \citep{Qaadan} or using low rank kernel linearizations \citep{Lan}.
Here we will concentrate, however, on the classical, kernel based dual approach to SVM training, seeking to accelerate the convergence of SMO in a similar way to how standard gradient descent has been accelerated in Deep Neural Network (DNN) training.

DNNs are currently the standard approach to big data problems.
One reason for their success is the skilled exploitation that has been made of several advances in optimization, often based in new ideas inspired by relatively simple techniques in convex optimization.
Gradient Descent (GD) can be analyzed with great precision on a purely convex setting \citep{NesterovIntrodLect} and the same is true of variants to make it faster.
Two well-known ways to improve on GD are the Heavy Ball method, a.k.a. momentum, a slightly coarser version of Conjugate Gradient (CG), 
and Nesterov's acceleration, in itself also a momentum-like method and routinely used for mini-batch gradient descent on DNNs.
%
While the analysis of the application of momentum or Nesterov's acceleration in DNNs can only partially replicate the precision that can be achieved in a pure convex setting, their simplicity makes it very easily to incorporate their basic ideas into other methods. 

From an optimization point of view, Gaussian SVM's dual problem is a quadratic programming problem with a positive definite matrix in most cases, an equality constraint and many simple box inequality ones.
Moreover, SMO can be seen as a projected approximate gradient descent algorithm and its iterations have a very simple and largely analytic structure that is amenable to a precise handling.
It is thus natural to study the possible application of some of the above convex optimization methods to accelerate SMO. We will consider in this work a CG variant tailored to SMO.

We point out that Nesterov's acceleration can also be easily adapted to the SMO algorithm.
However, and as discussed in \cite{torres2016nesterov}, each Nesterov iteration is considerably costlier than a pure SMO one, and the reduction in the number of iterations in Nesterov SMO is not enough to produce actually faster training times.
Our CG variant for SMO has a much smaller overhead, resulting not only in less iterations but also in actually fast training times that are at least competitive and often substantially faster than plain SMO.
This is the core of this work, which greatly expands a preliminary version in \cite{torres2016conjugate} and whose main contributions are:
\begin{enumerate}
\item A detailed proposal of a Conjugate SMO (CSMO) algorithm for SVM classification and regression, with a comprehensive complexity analysis.
\item Proofs of the convergence of CSMO for general kernel matrices and of its linear convergence for positive definite ones.
\item A detailed time comparison between CSMO and second order SMO, based on our implementation of CSMO inside the well-known and excellent LIBSVM library for kernel SVMs (implementation available on GitHub\footnote{\url{https://github.com/albertotb/libsvm_cd}}).
\end{enumerate}
%
The rest of the paper is organized as follows.
We will briefly review SVMs for classification and regression 
and SMO in \cref{sec:smo}.
Our CG-SMO algorithm will be described in \cref{sec:csmo}, which also contains the convergence proofs. Extensive experiments are presented in \cref{sec:exp}.
The paper ends with a discussion and pointers to further work.

\section{SVMs for Classification and Regression}
\label{sec:smo}
\subsection{Primal and Dual Problems}

Here we will work in the general setting introduced in \cite{Lin_smo_conv_wout_assump} that encompasses both
SV classification (SVC) and regression (SVR).
To begin, consider a set of triplets
$\Scal = \left\lbrace (\Xbf_i, y_i, s_i) : i = 1, \ldots, N \right\rbrace$
with $y_i = \pm 1$ and $s_i$ some scalar values,
and the following convex optimization problem:
\begin{equation}
\min_{\wbf, b, \xibf}\; \Pcal (\wbf, b, \xibf) = \frac{1}{2} \| {\wbf} \|^2 + C \sum_i \xi_i
\label{primalGen}
\end{equation}
subject to
$y_i (\wbf^\tr \Xbf_i + b) \geq s_i - \xi_i, \;\;\xi_i \geq 0, \;\; \forall i.$
When $s_i = 1$ this is just SVC, with $C$ a user-specified constant.
Similarly, for a sample
${\cal R} = \{ (\Xbf_i, t_i) : i = 1, \ldots, N \}$,
\cref{primalGen} reduces to $\epsilon$-insensitive SV regression (SVR) if we enlarge the sample to
$\Scal = \{ (X_i, y_i, s_i) : i = 1, \ldots, 2N \}$ taking
$y_i = 1$ and $s_i = t_i - \epsilon$ for $1 \leq i \leq N$,
and having $y_{N+i} = -1$, $s_{N+i} = -t_i - \epsilon$
and $X_{N+i} = X_i$.

Going through the Lagrangian of \eqref{primalGen} one arrives at its dual problem
\begin{align}
\min_{\alphabf}\; \Theta (\alphabf) &= \frac{1}{2} \sum_i \sum_j \alpha_i \alpha_j Q_{ij} - \sum_i \alpha_i s_i \nonumber \\
& =\frac{1}{2} \alphabf^\tr \Qbf \alphabf - \sbf^\tr \alphabf,
\label{dualGen}
\end{align}
where $Q_{ij} = y_i y_j \Xbf_i^\tr \Xbf_j$, and subject to the constraints
$$
0 \leq \alpha_i \leq C , \; 1 \leq i \leq N; \quad \sum_i \alpha_i y_i = 0.
$$
We shall refer to $\sum_i \alpha_i y_i = 0$ as the equality constraint.
Note that in a kernel setting we would replace the inner product $\Xbf_i^\tr \Xbf_j$ with a kernel function $K( \Xbf_i, \Xbf_j )$.

\subsection{Sequential Minimal Optimization}

The $\alphabf$ updates in the Sequential Minimal Optimization, SMO, are 
\begin{equation}\label{smo_updates}
\alphabf^{k+1} = \alphabf^{k} + \rho(y_l\ebf_l - y_u\ebf_u) = \alphabf^{k} + \rho \dbf
\end{equation}
where $\dbf = \dbf_{lu} = y_l\ebf_l - y_u \ebf_u$ is a descent vector, with $\ebf_k$ the vector with all zeros except a $1$ in the $k$-th entry.

We will use the notations $\dbf$ and $\dbf_{lu}$ indistinctly, dropping the subindices when there will be no confusion.
Notice that if  $\alphabf' = \alphabf_k + \rho \dbf$, the equality condition $\ybf^\tr \alphabf' = 0$ clearly holds for any $\rho$.
We have thus to choose $\dbf$ and $\rho$.
Starting with $\dbf$, note that the box constraints imply that the only eligible indices $(l,\, u)$ are those in the sets
\begin{align}
\Ical_\text{L} &\defeq \set{l \given \alpha_l < C,\, y_l = 1 \text{ or } \alpha_l > 0,\,y_l = -1}, 
\label{L_set} \\
\Ical_\text{U} &\defeq \set{u \given \alpha_u < C,\, y_u = -1 \text{ or } \alpha_u > 0,\,y_u = 1}. 
\label{U_set}
\end{align}
In \textbf{first order} SMO, the indices $L$, $U$ for $\dbf = \dbf_{LU}$ are chosen as 
\begin{align}
L & = \argmin_{l \in  \Ical_\text{L}}\; \{ y_l (\Qbf \alphabf^k)_l - y_l s_l \},
\label{L_choice} \\
U & = \argmax_{u \in  \Ical_\text{U}}\; \{ y_u (\Qbf \alphabf^k)_u - y_u s_u \};
\label{U_choice}
\end{align}
observe that then $\dbf_{LU} \cdot \nabla \Theta( \alphabf^k) = y_L (\Qbf \alphabf^k)_L - y_L s_L - y_U (\Qbf \alphabf^k)_U + y_U s_U  < 0$ and $\dbf_{LU}$ is thus a descent direction.

Next, we have to choose $\rho$.
It is easy to see that the unconstrained gain on $\Theta$ going from $\alphabf^k$ to $\alphabf' = \alphabf^{k} + \rho \dbf$ can be written as
\begin{equation} \label{dual_gain}
\Theta(\alphabf^k) - \Theta(\alphabf') =
-\frac{1}{2} \rho^2\, \dbf^\tr \Qbf \dbf - \rho\,\dbf^\tr (\Qbf \alphabf^k - \sbf) = \Psi(\rho),
\end{equation}
which has a maximum provided  $\dbf^\tr \Qbf \dbf > 0$.
Now the unconstrained maximum of \cref{dual_gain} is obtained by solving $\Psi'(\rho) = 0$, which results in
\begin{equation}\label{opt_rho_uncons}
\rho' = -\frac{\dbf^\tr \Qbf\alphabf^k - \dbf^\tr \sbf}{\dbf^\tr \Qbf \dbf}
= - \frac{\dbf^\tr \nabla \Theta(\alphabf^k)}{\dbf^\tr \Qbf \dbf};
\end{equation}
When this is inserted back  in \cref{dual_gain}, the dual gain becomes
\begin{align}
\Theta(\alphabf^{k}) - \Theta(\alphabf') 
& =
\frac{1}{2}\frac{(\dbf^\tr (\Qbf \alphabf^k - \sbf))^2}{\dbf^\tr \Qbf \dbf} \nonumber \\
& = \frac{1}{2}\frac{ (y_L (\Qbf \alphabf^k)_L - y_L s_L - (y_U (\Qbf \alphabf^k)_U - y_U s_U ))^2}{\dbf^\tr \Qbf \dbf}.
\label{dual_gain_2}
\end{align}
To get the optimum constrained step $\rho^*$ we clip $\rho'$ as
\begin{equation}\label{clip_smo}
{\rho}^\opt = \max\set{\min\set{\rho', \,-y_L(C - \alpha_L),\,y_U(C - \alpha_U)}, \, y_L\alpha_L,\, -y_U\alpha_U },
\end{equation}
yielding the final SMO updates
\begin{equation}
\alpha^{k+1}_L = \alpha^k_L + y_L{\rho}^\opt, \quad
\alpha^{k+1}_U = \alpha^k_U - y_U{\rho}^\opt.
\label{LU_update}
\end{equation}

Obviously, the first order $L, U$ choices maximize the numerator in \cref{dual_gain_2};
however, they also influence the denominator $\dbf_{LU}^\tr \Qbf \dbf_{LU}$.
The {\bf second order} SMO updates exploit this by choosing $L$ as in \cref{LU_update} but, once fixed, $U$ is selected as
\begin{equation}
U = \argmax_{u \in  \Ical_\text{U}} \left\{
\frac{ (y_L (\Qbf \alphabf^k)_L - y_L s_L - (y_u (\Qbf \alphabf^k)_u - y_u s_u ))^2}{\dbf_{Lu}^\tr \Qbf \dbf_{Lu}}
\right\}.
\label{U_choice_2}
\end{equation}
The resulting $d_{LU}$ gives again a descent direction but now with a larger unclipped gain on $\Theta$.
In other words, $\dbf_{LU}$ is a simple proxy of the full gradient $\nabla \Theta(\alphabf^k)$ that yields a greater gain.
We shall use later on the notations $\dbf_{LU} = \dbf(\alphabf^k) = \dbf^k$ and also 
$$\Delta(\alphabf^k) = y_U (\Qbf \alphabf^k)_U - y_U s_U - \left(y_L (\Qbf \alphabf^k)_L - y_L s_L \right) =- \nabla \Theta(\alphabf) \cdot \dbf_{LU}.$$
%

A consequence of the Karush-Kuhn-Tucker (KKT) conditions for the SVM primal and dual problems is that $\alphabf$ is a dual optimum if and only if $\Delta(\alphabf) \leq 0$.
Hence, if $\Delta(\alphabf) > 0$, there is at least one pair $l, u$ that violates this minimum condition;
in particular, the $L, U$ chosen for the first order SMO iterations are called a maximal violating pair.  
The SMO iterates continue until some stopping condition is met; the usual choice is to have
$\Delta (\alphabf^k) < \epsilon_{\text{KKT}}$ for some pre-selected KKT tolerance $\epsilon_{\text{KKT}}$. The whole procedure is summarized in Algorithm \ref{alg:smo}.

\begin{algorithm}[!t]
\DontPrintSemicolon
\SetKwInOut{Input}{Input}
\Input{$\alphabf=0 \in \Rbb^d$ and $C$}
\While{stopping condition not met} {
  Select working set $(L,U)$ using \eqref{L_choice}, \eqref{U_choice_2}\;
  Compute unconstrained stepsize $\rho$ as in \cref{opt_rho_uncons} \;
  Clip the stepsize if necessary as in \cref{clip_smo} \;
  $\alpha_L \leftarrow \alpha_L + y_L\hat{\rho}^\opt$ \;
  $\alpha_U \leftarrow \alpha_U - y_U\hat{\rho}^\opt$ \;
  Update the gradient at $\alpha$ \;
}
\caption{Sequential Minimal Optimization (SMO)}\label{alg:smo}
\end{algorithm}

\subsection{Cost and Convergence of SMO}

The cost per iteration of SMO is determined by the choice of $L$ and $U$, and the update of the gradient $\nabla \Theta(\alphabf^k)$.
Selecting $U$ requires $2N$ products, with $N$ being the sample size. To compute the gradient efficiently just note that
\begin{align}
\nabla \Theta(\alphabf + \rho \dbf)
 &= \Qbf(\alphabf + \rho \dbf) - \sbf = \nabla \Theta(\alphabf) + \rho \Qbf \dbf \notag \\
 &= \nabla \Theta(\alphabf) + \rho(y_L \Qbf_L - y_U \Qbf_U) \label{smo_grad_update},
\end{align}
where $\Qbf_k$ is the $k$-th column of the matrix $\Qbf$. Thus, a vector with the current gradient is maintained during the optimization and updated with a cost of $N$ products. 
In total, $3N$ floating point products are needed for each SMO update.

For a general positive semidefinite kernel matrix $\Qbf$, the dual problem \eqref{dualGen} does not have a unique solution, but the sequence $\alphabf^k$ of either first or second order SMO iterates has a subsequence that converges to a dual minimum $\alphabf^*$; see \cite{fan2005working} and \cite{chen2006study} for details.
However the primal problem \eqref{primalGen} has a unique minimum $\wbf^*$ and, when formulated in the reproducible Hilbert kernel space (RKHS) $H$ induced by the kernel $K$, the sequence
$\wbf^k = \sum_p \alpha^k_p y_p \Phi(\Xbf_p) \in H$ derived from the entire SMO iterate sequence converges to $\wbf^*$, 
where $\Phi(\Xbf)$ denotes the mapping induced by $K$ of the initial $\Xbf$ patterns into the RKHS $H$; see \cite{LopezDorronConvSMO}.

When $\Qbf$ is positive definite, there is a unique dual minimum $\alphabf^*$, the entire SMO sequence $\alphabf^k$ tends to $\alphabf^*$ and linear covergence of SMO has been proved under different assumptions (see for instance \cite{list2007general}).
We shall consider here the non-degeneracy condition in \cite{fan2005working}.
Let
\begin{equation}
{\cal H} = \{ q : y_q (\wbf^* \cdot \Xbf_q + b^*) = s_q \};
\label{h_def}
\end{equation}
%
then $s_q = y_q (\wbf^* \cdot \Xbf_q + b^*) = (\Qbf \alphabf^*)_q + y_q b^*$, i.e., 
$y_q s_q = y_q (\Qbf \alphabf^*)_q +  b^*$ for all $q \in {\cal H}$ and, therefore,
\begin{align}
\Delta(\alphabf^k) &= y_U (\Qbf \alphabf^k)_U - y_L (\Qbf \alphabf^k)_L - (y_U s_U - y_L s_L) \nonumber \\
&= y_U (\Qbf \alphabf^k)_U - y_L (\Qbf \alphabf^k)_L - \left(y_U (\Qbf \alphabf^*)_U - y_L (\Qbf \alphabf^*)_L \right) \nonumber \\
&= -(\alphabf^k - \alphabf^*) \cdot \Qbf \dbf_{LU} .
\label{delta_2}
\end{align}
The non-degeneracy condition is $q \in \cal{H}$ if and only if $0 < \alphabf^*_q < C$.
Then, it is shown in \cite{chen2006study}, Theorem 6, that for the first order SMO iterates, there is a $K >0$ and $c$, $0 < c < 1$, such that for all $k \geq K$, 
$$\Theta(\alphabf^{k+1}) - \Theta(\alphabf^\opt) \leq c(\Theta(\alphabf^k) - \Theta(\alphabf^\opt)).$$
We point out that Theorem 6 in \cite{chen2006study} considers more general SMO iterates but not second order ones.

\section{Conjugate SMO}
\label{sec:csmo}
\subsection{Conjugate Directions for SMO}

Recall that SMO updates are of the form
$\alphabf^{k+1} = \alphabf^{k} + \rho_k \dbf^k$
where $\dbf^k = y_L\ebf_L - y_U \ebf_U$ and $(L,U)$ are the indices selected by the SMO procedure described in \cref{sec:smo}. 

Following \cite{torres2017phd}, we replace the descent direction $\dbf^k$ by an appropriate conjugate direction $\pbf^k$, $\alphabf^{k+1} = \alphabf^k + \rho_k \pbf^k$ with
\begin{equation}\label{cdsmo_p_dir}
\pbf^k = \dbf^k + \gamma_k \pbf^{k-1}
\end{equation}
and where $\dbf^k$ is chosen as in standard first or second order SMO at $\alphabf^k$.
If the preceding $\pbf^{k-1}$ verifies $\sum{y_i p^{k-1}_i} = 0$, then
$$\sum{y_i p^k_i} = \sum{y_i (d^k_i + \gamma_k p^{k-1}_i)} = 0$$
and the new $\alphabf^k$ automatically verifies the linear constraint, i.e.,
$$\sum{y_i \alpha^{k+1}_i} = \sum{y_i \alpha^k_i} + \rho_k \sum{y_i p^k_i} = 0.$$
Now, assume for the time being that $\gamma_k$ and the conjugate direction $\pbf^k$ have been chosen; we then find the unconstrained ${\rho}_k$ factor by minimizing $\Theta$ along $\pbf^k$. 
Let $\gbf^k = \nabla \Theta(\alphabf^k) = \Qbf \alphabf^k - \sbf$ be the gradient of the SVM objective function at $\alphabf^k$; then we have
\begin{align*}
\frac{\partial}{\partial {\rho}} \Theta (\alphabf^k + {\rho} \pbf^k)
&= {\rho} \pbf^k \cdot \Qbf \pbf^k + \pbf^k \Qbf \alphabf^k - \pbf^k \cdot \sbf \\
&= {\rho} \pbf^k \cdot \Qbf \pbf^k + \pbf^k (\Qbf \alphabf^k - \sbf) \\
&= {\rho}_k\pbf^k \cdot \Qbf \pbf^k + \gbf^k \cdot \pbf^k.
\end{align*}
Writing $\pbf^k$ in terms of $\pbf^{k-1}$, 
we solve $\frac{\partial}{\partial {\rho}} \Theta (\alphabf^k + {\rho} \pbf^k) = 0$ by taking
\begin{equation}\label{cdsmo_rho_opt1}
{\rho}_k^\opt = \frac{ -\gbf^k \cdot \pbf^k}{\pbf^k \cdot \Qbf \pbf^k} = \frac{ -\gbf^k \cdot (\dbf^k + \gamma_k\cdot \pbf^{k-1})}{\pbf^k \cdot \Qbf \pbf^k} = \frac{ -\gbf^k \cdot \dbf^k - \gamma_k \gbf^k \cdot \pbf^{k-1}}{\pbf^k \cdot \Qbf \pbf^k}.
\end{equation}
%
%
If the previous line minimization along $\pbf^{k-1}$ has been unclipped, i.e. we have $\alphabf^k = \alphabf^{k-1} + \rho_{k-1} \pbf^{k-1}$, then $\alphabf^k$ is the optimum of $\Theta$ along the line $\alphabf^{k-1} + \rho \pbf^{k-1}$ and hence $\nabla \Theta(\alphabf^k)$ and $\pbf^{k-1}$ are orthogonal, i.e., the following condition must hold:
\begin{equation} \label{1srtOrthog}
\gbf^k \cdot \pbf^{k-1} = 0.
\end{equation}
We will call \cref{1srtOrthog} the {\bf first orthogonality condition}.
As a consequence,
$$\gbf^k \cdot \pbf^k = \gbf^k \cdot \dbf^k + \gamma_k \gbf^k \cdot \pbf^{k-1} = \gbf^k \cdot \dbf^k  < 0, $$
i.e. $\pbf^k$ is a descent direction, since so is $\dbf^k$. 
Besides, \cref{cdsmo_rho_opt1} simplifies to
\begin{equation}\label{cdsmo_rho_opt2}
{\rho}_k^\opt = -\frac{\gbf^k \cdot \dbf^k}{\pbf^k \cdot \Qbf \pbf^k},
\end{equation}
and is easy to see that the unconstrained gain in $\Theta$ is now 
%
\begin{equation}\label{conjugate_dual_gain}
\Theta(\alphabf^k) - \Theta(\alphabf') = \frac{1}{2} \frac{ (\gbf^k \cdot \dbf^k)^2}{\pbf^k \cdot \Qbf \pbf^k}.
\end{equation}
Just as before, our choice of $\dbf^k$ may maximize the numerator but now we can further maximize on this gain by choosing $\gamma_k$ to minimize the denominator
${\pbf^k \cdot \Qbf \pbf^k}$. Writing it as a function of $\gamma$, we have
\begin{align*}
\phi(\gamma) &=(\dbf^k + \gamma \pbf^{k-1}) \cdot \Qbf (\dbf^k + \gamma \pbf^{k-1}),\\
\phi'(\gamma) &= 2(\dbf^k \cdot \Qbf \pbf^{k-1} + \gamma \pbf^{k-1} \cdot \Qbf \pbf^{k-1});
\end{align*}
thus, solving $\phi'(\gamma)=0$ yields 
\begin{equation} \label{cdsmo_gamma_opt}
\gamma_k^\opt = - \frac{\dbf^k \cdot \Qbf \pbf^{k-1}} {\pbf^{k-1} \cdot \Qbf \pbf^{k-1}}.
\end{equation}
Now it is easy to see that this choice results in
%
$\pbf^k \cdot \Qbf \pbf^{k-1} = 0$;
%
we will call this equation
the \textbf{second orthogonality condition}. 
Plugging this $\gamma_k^\opt$ estimate into the $\pbf^k$ vector in the denominator of \cref{conjugate_dual_gain}, it becomes
\begin{equation*}
{\pbf^k \cdot \Qbf \pbf^k} = \dbf^k \cdot \Qbf \dbf^k - \frac{ (\dbf^k \cdot \Qbf \pbf^{k-1})^2} {\pbf^{k-1} \cdot \Qbf \pbf^{k-1}} = 
\|\dbf^k\|_\Qbf^2 \left(1 - \frac{(\dbf^k \cdot \Qbf \pbf^{k-1})^2}{ \|\dbf^k\|_\Qbf^2 \|\pbf^{k-1}\|_\Qbf^2} \right)    
\end{equation*}
%
where we use the notation $\|u\|_\Qbf^2 =  u \cdot \Qbf u$.
Therefore, we can write the unclipped gain in \eqref{conjugate_dual_gain} as
\begin{equation}\label{conjugate_dual_gain_2}
\Theta(\alphabf^k) - \Theta(\alphabf') = \frac{1}{2} \frac{ (\gbf^k \cdot \dbf^k)^2}
{ \|\dbf^k\|_\Qbf^2 \left(1 - \frac{(\dbf^k \cdot \Qbf \pbf^{k-1})^2}{ \|\dbf^k\|_\Qbf^2 \|\pbf^{k-1}\|_\Qbf^2} \right)}.
\end{equation}
Since $\dbf^k \cdot \Qbf \pbf^{k-1} \leq \|\dbf^k\|_\Qbf \|\pbf^{k-1}\|_\Qbf$, 
it follows that the unclipped CSMO gain in \cref{conjugate_dual_gain_2} is always larger than the unclipped SMO gain in \cref{dual_gain_2} at $\alphabf^k$.

We can summarize now our conjugate SMO updates. If the iteration ending in $\alphabf^k$ along $\pbf^{k-1}$ has not been clipped, we:
\begin{enumerate}
\item update $\pbf^k$ from $\pbf^{k-1}$ using \cref{cdsmo_p_dir} and then $\gamma_k$ using \cref{cdsmo_gamma_opt},
\item compute ${\rho}_k$ and $\alphabf'$ using \cref{cdsmo_rho_opt2} 
and, finally,
\item check whether $\alphabf'$ satisfies the box constraints $$0 \leq \alphabf'_i \leq C$$ and, if not, clip ${\rho}_k$ accordingly to get the final $\rho_k$ and to arrive at 
$\alphabf^{k+1}$.
\end{enumerate}
%
%
We point out that having to clip the $\alphabf^{k+1}$ update implies that we have hit the boundary of the box region.
When this happens, we will simply reset $\pbf^k$ to 0 after the update, as keeping the current conjugate direction may lead to further boundary hits. 
We will then have $\pbf^{k+1} = \dbf^{k+1}$ at the new iteration, which becomes a standard SMO update.

\subsection{Efficient Conjugate SMO}

At first sight, the possible advantages of working with the conjugate directions may be offset by their cost, higher than that of the  SMO iterations.
However, working with the SMO descent directions $\dbf_{L U} = y_L \ebf_L - y_U \ebf_U$ greatly simplifies these computations. For this, we will keep an auxiliary vector $\qbf = \Qbf \pbf$ and constant $\delta = \pbf \cdot \Qbf \pbf$ that we will update at each iteration and use them to simplify the computation of the other elements as follows:
\begin{align}
\begin{split}
\gamma_k &= - \frac{\dbf^k \cdot \Qbf \pbf^{k-1}} {\pbf^{k-1} \cdot \Qbf \pbf^{k-1}} = - \frac{\dbf^k \cdot \qbf^{k-1}} {\delta^{k-1}}\\
             &= \frac{y_U \qbf^{k-1}_U - y_L \qbf^{k-1}_L} {\delta^{k-1}},
\end{split}  \label{cdsmo_gamma_k} \\[6pt]
\begin{split}
\qbf^k  &=  \Qbf (\dbf^k + \gamma_k \pbf^{k-1}) \\
        &= y_L \Qbf_L  - y_U \Qbf_U  + \gamma_k \qbf^{k-1},
\end{split} \label{cdsmo_q_k} \\[6pt]
\begin{split}
\delta_k &=  \pbf^k \cdot \Qbf \pbf^k = \dbf^k \cdot \Qbf \pbf^k = \dbf^k \cdot \qbf^k  \\
         &= y_L \qbf^k_L  - y_U \qbf^k_U,
\end{split} \label{cdsmo_delta_k} \\[6pt]
{\rho}_k  &=  \frac{ -\gbf^k \cdot \dbf^k}{\delta_k} = \frac{ y_{U} \gbf^k_{U}- y_{L} \gbf^k_{L} }{\delta_k}, \label{cdsmo_rho_k}
\end{align}
where $\Qbf_j$ denotes $\Qbf$'s $j$th column. Note that using \cref{cdsmo_q_k} the gradient can be efficiently updated as
\begin{equation}\label{cdsmo_g_k}
\gbf^{k+1} = \nabla \Theta(\alphabf^k + \rho_k \pbf^k) = \Qbf\alphabf^k + \rho_k\Qbf \pbf^k - \sbf = \gbf^k + \rho_k \qbf^k.
\end{equation}

\begin{algorithm}[t]
\DontPrintSemicolon
\LinesNumbered
\SetKwInput{Input}{Input}
\SetKwInOut{Initialize}{Initialize}
\SetKw{KwGoTo}{go to}
\Initialize{$\alphabf^0=\gbf^0=\zeros, \pbf^{-1}=\qbf^{-1}=\zeros$, $\delta^{-1} = 1$}
\While{stopping condition not met} {
   Select working set $(L,U)$ using \eqref{L_choice}, \eqref{U_choice_2}\; \label{cdsmo_selectLU}
   Compute the kernel matrix columns $\Qbf_L$ and $\Qbf_U$, if not previously cached\;
   $\gamma_k = (y_U \qbf^{k-1}_U - y_L \qbf^{k-1}_L)/\delta^{k-1}$\;
   $\pbf^k = \dbf^k + \gamma_k \pbf^{k-1}$\; \label{cdsmo_updateP}
   $\qbf^k = y_L \Qbf_L  - y_U \Qbf_U  + \gamma_k \qbf^{k-1}$\; \label{cdsmo_updateQ}
   $\delta^k = y_L \qbf^k_L  - y_U \qbf^k_U$\;
   Compute ${\rho}_k$ as in \eqref{cdsmo_rho_k} and clip it if needed \; \label{cdsmo_stepAndClip}
   $\alphabf^{k+1} = \alphabf^k + \rho_k \pbf^k$\; \label{cdsmo_updateA}
   $\gbf^{k+1} = \gbf^k + \rho_k \qbf^k$\; \label{cdsmo_updateG}
   \If{$\tilde{\rho}_k$ was clipped}{
     $\qbf^k = \pbf^k = \zeros$\;
     $\delta_k = 1$\;
   }
}
\caption{Conjugate SMO (CSMO)}\label{alg:cdsmo}
\end{algorithm}

The pseudocode for the conjugate version of SMO is shown in Algorithm \ref{alg:cdsmo}. Regarding the clipping of ${\rho}_k$, we need a value that ensures
$0 \leq \alpha^k_j + \rho_k p^k_j \leq C$  for all $j$.
Let us define the index sets
$$\Pcal^k_+  = \set{i \given p^k_i > 0} \quad \text{and} \quad \Pcal^k_-  = \set{i \given p^k_i < 0}.$$
Thus, to make sure that the following inequalities hold
\begin{alignat*}{3}
0 < &\; \rho_k \leq &&\; \frac{C-\alpha^k_i}{p^k_i} \quad &&\text{if }  i \in \Pcal^k_+, \\
0 < &\; \rho_k \leq &&\; -\frac{\alpha^k_i}{p^k_i} \quad  &&\text{if }  i \in \Pcal^k_-,
\end{alignat*}
we take the new $\rho_k$ as $\min \set{{\rho}_k,\, \rho^+,\, \rho^-}$, where
$$\rho^+ = \min \set{\frac{C-\alpha^k_i}{p^k_i} \given i \in \Pcal^k_+ }, \;\;
\rho^- = \min \set{-\frac{\alpha^k_i}{p^k_i} \given i \in \Pcal^k_- }.$$

We finish this section discussing the computational cost of the conjugate SMO updates. 
If $N_p$ and $N_q$ denote the number of non-zero components of $\pbf$ and $\qbf$ respectively, the cost in products of each iteration is
\begin{enumerate}
\item $2N$ products  in line \ref{cdsmo_selectLU} of Algorithm \ref{alg:cdsmo} when selecting $L$ and $U$;
\item $N_p$ products to update $\pbf$ in line \ref{cdsmo_updateP}, to update $\alphabf$ in line \ref{cdsmo_updateA}, and to compute a clipped $\rho_k$ in line \ref{cdsmo_stepAndClip};
\item $N_q$ products  to update $\qbf$ in line \ref{cdsmo_updateQ} and
\item $N$ products to update the gradient $\gbf$ in line \ref{cdsmo_updateG}.
\end{enumerate}
We expect $N_q \simeq N$ but $N_p$ should coincide with the number of non-zero components in $\alphabf$; this number should be $\ll N$ and, similarly,
we should have $N_p \ll N$. Thus a conjugate iteration should theoretically add a cost of
$$3N + N_q + 3 N_p \simeq 3N + N_q \simeq 4N,$$
in contrast with $3N$ for a standard SMO iteration. Therefore, CSMO should lead to a faster training if the number of SMO iterations is more than $4/3$ the number of CSMO ones. In any case, note that the cost of the iterations in which a non-cached kernel column matrix has to be computed will require a much larger $N \times d$ number of products when working with patterns in a $d$-dimensional space or even more in a kernel setting. Thus, in the starting iterations the cost of the SMO and conjugate SMO would be dominated by the much larger cost of computing the required $\Qbf$ columns.

\subsection{Convergence of Conjugate SMO}
We show first the convergence of CSMO.
\begin{proposition}
There is a subsequence of CSMO iterates that converge to a dual minimum.
\end{proposition}
\begin{proof}
We adapt the argument for SMO given in \cite{LopezDorronConvSMO}.
Let us denote by $U$ and $\tilde{U}$ the $U$  choices of first and second order SMO (the $L$ choices coincide).
First, observe that since the second order gain is larger than the first order one, we must have 

$$\frac{ (\gbf^k \cdot \dbf^k_{L \tilde{U}})^2} { \|\dbf^k_{L \tilde{U}}\|_\Qbf^2 } \geq \frac{ (\gbf^k \cdot \dbf^k_{L U})^2} { \|\dbf^k_{L U}\|_\Qbf^2 } = 
\frac{ \Delta(\alphabf^k)^2} { \|\dbf^k_{L {U}}\|_\Qbf^2 } .$$
As a consequence, for an unclipped CSMO iteration it follows from \cref{conjugate_dual_gain_2} that 
\begin{equation}\label{conjugate_dual_gain_3}
\Theta(\alphabf^k) - \Theta(\alphabf^{k+1}) 
\geq \frac{1}{2} \frac{ (\gbf^k \cdot \dbf^k_{L \tilde{U}})^2} { \|\dbf^k_{L \tilde{U}}\|_\Qbf^2 } 
\geq \frac{1}{2} \frac{ \Delta(\alphabf_k)^2} { \|\dbf^k_{L {U}}\|_\Qbf^2 } 
\geq \kappa \Delta(\alphabf_k)^2,
\end{equation}
with $1 / \kappa$ an upper bound of ${2 \|\dbf_{L U}^k\|_\Qbf^2}$.
Now, it is shown in \cite{LopezDorronConvSMO} that there is a maximum number $M'$ of consecutive clipped 
SMO iterations. 
Since we restart CSMO with a plain SMO iteration after a clipped CSMO one, a sequence of $m$ consecutive clipped CSMO iterations is made of a single CSMO one and $m-1$ clipped SMO iterations afterwards.
Thus, there is a maximum number $M = M'+1$ of consecutive clipped CSMO iterations.
In particular, there must be a subsequence $k_j$ of unclipped CSMO iterations and, since the $\Theta(\alphabf^k)$ sequence is decreasing and bounded from below, it follows from \eqref{conjugate_dual_gain_3} that $\Delta_{k_j} = \Delta(\alphabf^{k_j}) \rightarrow 0$.

Now, since the subsequence $\alphabf^{k_j}$ is bounded, it contains another subsequence, which we will also denote as $\alphabf^{k_j}$, which converges to a feasible $\overline{\alphabf}$.
Assume $\overline{\alphabf}$ is not a dual minimum, i.e., $\overline{\Delta} = \Delta(\overline{\alphabf}) > 0$, and let $L,U$ be a most violating first order SMO pair for $\overline{\alphabf}$;
then, as shown in~\cite{LopezDorronConvSMO}, Proposition 4, there is a $K_0$ such that for all $k_j \geq K_0$, $L, U$ is an eligible pair for $\alphabf^{k_j}$ and, also,
\begin{equation}
| y^q \nabla \Theta (\alphabf^{k_j})_q - y^q \nabla \Theta (\overline{\alphabf})_q| \leq \frac{\overline{\Delta}}{4}.
\label{unif_conv}
\end{equation}
%
We now have
\begin{align*}
\Delta_{k_j} & \geq y_U \nabla \Theta (\alphabf^{k_j})_U - y_L \nabla \Theta (\alphabf^{k_j})_L \\
             & \geq y_U \nabla \Theta (\overline{\alphabf})_U - y_L \nabla \Theta (\overline{\alphabf})_L - 2 \frac{\overline{\Delta}}{4} = \Delta(\overline{\alphabf}) - \frac{\overline{\Delta}}{2} = \frac{\overline{\Delta}}{2},
\end{align*}
where the first inequality follows from the eligibility of $L, U$ for $\alphabf^{k_j}$ and the second from \eqref{unif_conv}.
But $\Delta_{k_j} \geq \overline{\Delta}/2 >0$ contradicts $\Delta_{k_j} \rightarrow 0$.
Thus, $\Delta(\overline{\alphabf}) \leq 0$ and the sequence $\alphabf^{k_j}$ converges to the dual minimum $\overline{\alphabf}$.
\end{proof}

Observe that the above proof works when using in CSMO either first or second order SMO updates. 
An easy consequence of this is that $\Theta(\alphabf^{k_j}) \rightarrow \Theta(\overline{\alphabf})$ and, thus, the entire sequence $\Theta(\alphabf^{k})$ converges to the dual minimum.
Moreover, and as pointed out for SMO, the sequence $\wbf^k = \sum_q \alphabf^k_q y_q \Phi(\Xbf_q)$ converges to the unique primal minimum $\wbf^*$.
We will show next linear convergence of first order CSMO updates  assuming $\Qbf$ to be positive definite and the same non-degeneracy condition used for linear convergence in SMO.
\begin{proposition}
Assume $\Qbf$ to be positive definite and the non-degeneracy condition that $q \in \cal{H}$ if and only if $0 < \alphabf^*_q < C$, with $\alphabf^*$ the unique dual minimum and ${\cal H}$ defined in \eqref{h_def}.
Then, when using first order SMO updates, the entire CSMO sequence $\alphabf^k$ converges linearly to $\alphabf^*$.
\label{csmo_lin_conv_full}
\end{proposition}
The arguments below rely on the ideas in Section 5 of~\cite{lopez2015linear}. 
The key result is  the following.
\begin{proposition}
Under the previous assumptions, there is a $\eta > 0$ and $K$ such that for all $k \geq K$,
\begin{equation}
(\alphabf^k - \alphabf^*)^\tr \Qbf \dbf(\alphabf^k) \geq \eta  \|\alphabf^k - \alphabf^*)\|_\Qbf  \; \|\dbf(\alphabf^k)\|_\Qbf.
\label{lower_bound}
\end{equation}
\label{csmo_lin_conv}
\end{proposition}
Once Proposition \ref{csmo_lin_conv} is proved, the proof of Proposition \ref{csmo_lin_conv_full} follows by easy modifications of Theorems 4 and 5 in~\cite{lopez2015linear}. 
We prove Proposition \ref{csmo_lin_conv} next.
\begin{proof}
Let $\wbf^*$ be the primal minimum and consider, besides ${\cal H}$, the index sets
\begin{align*}
{\cal O} & = \{q : y_q(\wbf^* \cdot \Xbf_q + b^*) < s_q\}, \\
{\cal B} & = \{q : y_q(\wbf^* \cdot \Xbf_q + b^*) > s_q\}.
\end{align*}
A consequence of the convergence of the primal CSMO iterates to the primal minimum $\wbf^*$ is that there is a $K_0$ such that for all $k \geq K_0$, $\alphabf^k_q = 0$ if $q \in {\cal O}$ and $\alphabf^k_q = C$ if $q \in {\cal B}$. 
This can be proved along the lines of Theorem 1 in \cite{lopez2015linear}.
A first consequence of this is that 
$\sum_{\cal O} \alphabf^k_p y_p = 0$ and $\sum_{\cal B} \alphabf^k_p y_p = C \sum_{\cal B} y_p$ for all $k \geq K_0$ and, also, for $\alphabf^*$.
Moreover, and  again for all $k \geq K_0$, the indices $L, U$ for the SMO updates at $\alphabf^k$ must be selected over the pairs $(l, u) \in {\cal H} \times {\cal H}$.
%
%
From now on, we will sometimes write  $\dbf(\alphabf^k)$ instead of $\dbf^k$ for a clearer understanding.

If \eqref{lower_bound} does not hold, we must have 
$$\lim \; \inf \frac{(\alphabf^k - \alphabf^*)^\tr \Qbf \dbf(\alphabf^k)}{\|\alphabf^k - \alphabf^*)\|_\Qbf  \; \|\dbf(\alphabf^k)\|_\Qbf} = 0.$$
Write $\vbf^k = \frac{\alphabf^k - \alphabf^*} {\|\alphabf^k - \alphabf^*\|_\Qbf}$;
then $\|\vbf^k\|_\Qbf = 1$ and there is a subsequence $k_j$ such that $\vbf^{k_j} \cdot \Qbf \dbf(\alphabf^{k_j}) \rightarrow 0$.
Moreover, since the $\vbf^{k_j}$ and $\dbf(\alphabf^{k_j})$ sequences are bounded, we can find a new subsequence, which we also denote as $k_j$, such that 
$\vbf^{k_j} \rightarrow \vbf$ and $\dbf(\alphabf^{k_j}) \rightarrow \dbf$ for some appropriate $\vbf$ and $\dbf$.
Furthermore, since the $\dbf(\alphabf^{k_j})$ only have two nonzero $\pm 1$ components, the same must be true for $\dbf$;
thus we can assume $\dbf(\alphabf^{k_j}) = \dbf = \dbf_{L, U}$ for some $L, U \in {\cal H} \times {\cal H}$ and for all $k_j > K_0$.

We will show next that such a $\vbf$ must be 0, contradicting that $\|\vbf\|_Q = \lim \|\vbf^{k_j}\|_\Qbf = 1$.
Now, 
the convergence of the $\alphabf^k$ to $\alphabf^*$ and the non-degeneracy assumption imply that for some $K_1$ and all $k \geq K_1$, we have $0 <  \alpha^k_q <C$ for all $q \in {\cal H}$.
It thus follows that any pair $(p, q) \in {\cal H} \times {\cal H}$ is eligible for any $\alphabf^{k_j}$.
But since $L, U$ is a maximum violating pair at $\alphabf^{k_j}$, this implies 
$$-(\alphabf^{k_j} - \alphabf^*)^\tr \Qbf \dbf_{pq} = \Delta(\dbf_{pq}) \leq \Delta(\alphabf^{k_j}) = -(\alphabf^{k_j} - \alphabf^*)^\tr \Qbf \dbf(\alphabf^{k_j}),$$
where, in a slight abuse of the notation, we write 
$\Delta(\dbf_{pq}) = \nabla \Theta(\alphabf^k) \cdot \dbf_{pq}$
and we have used \eqref{delta_2} in the last equality.
%
Taking limits, we would have 
$$-\vbf^\tr \Qbf \dbf_{pq} = \lim \; -\vbf^{k_j} \cdot \Qbf \dbf_{pq} \leq \lim \; -\vbf^{k_j} \cdot \Qbf \dbf^{k_j} = -\vbf^\tr \Qbf \dbf = 0,$$
or, in other words, $y^q(\Qbf \vbf)_q - y^p(\Qbf \vbf)_p \leq 0$ for all $(p, q) \in {\cal H} \times {\cal H}$.
But reversing the roles of $p$ and $q$, this also implies $y^q(\Qbf \vbf)_q = y^p(\Qbf \vbf)_p$; 
in other words, there is a $\nu \neq 0$ such that 
$(\Qbf \vbf)_q = \nu = y^q \nu$ for all $q \in {\cal H}$ such that $y^q = 1$, 
and, similarly, $(\Qbf \vbf)_q = -\nu = y^q \nu$ for all $q \in {\cal H}$ such that $y^q = -1$.

We finish the proof by showing that the preceding implies $(\alphabf^{k_j} - \alphabf^*) \cdot \Qbf \vbf \rightarrow 0$.
First, taking $K_2 = \max (K_0,K_1)$, since $\alphabf^{k_j}_q - \alphabf^*_q = 0$ if $q \in {\cal O} \cup {\cal B}$, we must have 
$$(\alphabf^{k_j} - \alphabf^*) \cdot \Qbf \vbf = \sum_{q \in {\cal H}} (\alphabf^{k_j}_q - \alphabf^*_q) (\Qbf \vbf)_q
= \nu  \sum_{q \in {\cal H}} (\alphabf^{k_j}_q - \alphabf^*_q) y^q$$
for all $k_j \geq K_2$.
But we also have $\sum_q \alphabf^{k_j}_q y^q = 0$, i.e.
$$
0 = \sum_{q \in {\cal O} \cup {\cal B} \cup {\cal H}} \alphabf^{k_j}_q y^q = 
C \sum_{q \in {\cal B} } y^q  +  \sum_{q \in{\cal H}} \alphabf^{k_j}_q y^q ,
$$
and, similarly, 
$$0 = \sum_{q \in {\cal O} \cup {\cal B} \cup {\cal H}} \alphabf^*_q y^q = C \sum_{q \in {\cal B} } y^q  + \sum_{q \in{\cal H}} \alphabf^*_q y^q.$$
But this implies 
$$0 = \sum_q \alphabf^{k_j}_q y^q - \sum_q \alphabf^*_q y^q = \sum_{q \in{\cal H}} (\alphabf^{k_j}_q - \alphabf^*_q) y^q$$
and, therefore $0 = (\alphabf^{k_j} - \alphabf^*) \cdot \Qbf \vbf$, 
which implies $0 = \vbf^{k_j} \cdot \ \Qbf \vbf$ and, taking limits, $0 = \vbf \cdot \Qbf \vbf = \| \vbf \|_\Qbf$, a contradiction that ends the proof of \cref{csmo_lin_conv}.
\end{proof}

\section{Experiments}
\label{sec:exp}
\subsection{Datasets and CSMO Implementation}

The goal of this Section is to empirically compare the running times of standard second order SMO and its conjugate gradient counterpart over 12 relatively large classification and regression datasets and various $C$, $\gamma$ and  $\epsilon$ configurations.

\begin{table}[htbp]
\centering
\small
\begin{subtable}[b]{0.55\textwidth}
\centering
\begin{tabular}{l
S[table-format=6]
S[table-format=3]
S[table-format=5]
S[table-format=6]}
\toprule
{Dataset}          &  {$n$} & {$d$}& {$n^+$}& {$n^-$}\\
\midrule
\data{adult8}      &  22696 &  123 &   5506 &  17190 \\
\data{web8}        &  49749 &  300 &   1479 &  48270 \\
\data{ijcnn1}      &  49990 &   22 &   4853 &  45137 \\
\data{cod-rna}     &  59535 &    8 &  19845 &  39690 \\
\data{mnist1}      &  60000 &  784 &   6742 &  53258 \\
\data{skin}        & 245057 &    3 &  50859 & 194198 \\
\bottomrule
\end{tabular}
\caption{Classification.}
\end{subtable}
\begin{subtable}[b]{0.44\textwidth}
\centering
\begin{tabular}{l
S[table-format=6]
S[table-format=3]}
\toprule
{Dataset}          & {$n$} & {$d$} \\
\midrule
\data{abalone}     &  4177 &   8   \\
\data{cpusmall}    &  6143 &  12   \\
\data{trajectory}  & 20000 & 297   \\
\data{cadata}      & 20640 &   7   \\
\data{year}        & 46215 &  90   \\
\data{ctscan}      & 53500 & 385   \\
\bottomrule
\end{tabular}
\caption{Regression.}
\end{subtable}
\caption{Number of observations ($n$), dimensions ($d$) and class ratios ($n^+$, $n^-$) for the datasets considered.} \label{tab:datasets}
\end{table}

The datasets used for the experiments are summarized in \cref{tab:datasets}. All of them can be found in LIBSVM's data repository, except from \data{ctscan}, \data{year} (UCI Machine Learning Repository\footnote{\url{http://archive.ics.uci.edu/ml}}), and \data{trajectory}, \data{mnist1} (Machine Learning Dataset Repository\footnote{\url{http://mldata.org}}). The MNIST database consists on $28 \times 28$ images of handwritten digits that have been preprocessed and normalized. Since it is a multi-class classification problem we use here the binary version that tries to distinguish the number $1$ from the rest. We choose these datasets since they offer a wide variety of sample sizes and dimensions. 
We have run our experiments in Intel Xeon E5-2640 computer nodes  with \SI{2.60}{\giga\hertz}, 8 cores and \SI{32}{\giga\byte} RAM. Note that on such machines kernel SVMs are probably not going to scale well with sample sizes above 100K--200K.

For second order SMO we will use the C++ implementation in the LIBSVM library. 
We have implemented conjugate SMO versions for $C$-SVC  and $\epsilon$--SVR inside the same C++ LIBSVM code, and it is freely available on Github\footnote{\url{https://github.com/albertotb/libsvm_cd}}. The LIBSVM cache works perfectly with these options as well as most of the other options for running LIBSVM.
However, our Conjugate SMO implementation is not adapted to shrinking. We  have also not tested the probabilistic estimates for classification. 
Actual time comparisons are presented in \cref{time_comp} but before this we address in the next Subsection the effect of the cache in LIBSVM performance.

\subsection{Cache Effects}\label{sec:cache}

SVM training times depend obviously on the number of iterations needed to achieve a desired tolerance but also on the times each iteration requires.
These times are not homogeneous as they depend on the number of kernel operations (KOs) to be done at each iteration.
To minimize on this LIBSVM implements a cache where KO results are stored and retrieved when needed, provided of course they have been previously performed and are stil in the cache.
This has a large influence on iteration time and, hence, on the overall SVM training times.

Therefore, cache size may have an important influence when comparing standard and conjugate SMO times.
Since at each iteration conjugate SMO gives a larger decrease to the dual function, it requires less iterations than standard SMO.
If the cache is small, many of these iterations will require KOs. 
In turn, they will dominate iteration times giving an advantage to the method requiring less iterations (presumably, conjugate SMO).

However, if a large cache is used, it is usually filled in the initial iterations which update at most two rows per iteration.
This implies that standard and conjugate SMO will need a similar number of slower iterations to fill the cache and that the overall training time is likely to be dominated by that of the initial cache filling iterations.
This implies that since subsequent iterations will be much faster, the advantage of the method requiring less iterations is likely to be smaller.
In summary, when cache sizes are rather small, conjugate SMO should have clearly better running times but things should even up with large caches.

\begin{table}[thbp]
\centering
\footnotesize
\sisetup{detect-weight=true,detect-inline-weight=math}%
\begin{tabular}{lS[table-format=5]S[table-format=1.2]S[table-format=1.2]S[table-format=3.2]S[table-format=3.2]S[table-format=4.2]S[table-format=4.2]}
\toprule
    &      & \multicolumn{6}{c}{Time (s)} \\
\cmidrule(lr){3-8}
    &      & \multicolumn{2}{c}{$C=1$} & \multicolumn{2}{c}{$C=100$} & \multicolumn{2}{c}{$C=\num{10000}$} \\
\cmidrule(lr){3-4}\cmidrule(lr){5-6}\cmidrule(lr){7-8}
Dataset    & {Cache} & {SMO} & {CSMO} & {SMO} & {CSMO} & {SMO} & {CSMO} \\
\midrule
\data{adult8}     & 1     &           5.52 & \bfseries 5.29  &          100.39 & \bfseries 62.47  &          6216.00 & \bfseries 2619.61  \\
           & 50    &           4.83 & \bfseries 4.43  &           26.10 & \bfseries 21.49  &          3512.05 & \bfseries 1908.27  \\
           & 100   &  \bfseries 4.31  &          4.42 &           12.70 & \bfseries 11.41  &          2405.84 & \bfseries 1350.65  \\
           & 500   &  \bfseries 4.24  &          4.37 &            9.23 & \bfseries  8.34  &           318.93 & \bfseries  211.85  \\
           & 1000  &  \bfseries 4.07  &          4.22 &            8.67 & \bfseries  7.82  &           305.88 & \bfseries  197.70  \\
           & 5000  &           4.05 & \bfseries 4.04  &            8.68 & \bfseries  7.83  &           306.33 & \bfseries  196.32  \\
           & 10000 &  \bfseries 3.96  &          4.04 &            8.67 & \bfseries  7.82  &           305.95 & \bfseries  196.20  \\ [0.5em]
\data{cpusmall}   & 1     &           1.36 & \bfseries 1.10  &           68.83 & \bfseries 47.95  &          2960.25 & \bfseries 2066.70  \\
           & 50    &           0.64 & \bfseries 0.59  &           22.12 & \bfseries 16.96  &           944.43 & \bfseries  773.95  \\
           & 100   &           0.64 & \bfseries 0.58  &           22.08 & \bfseries 16.89  &           916.96 & \bfseries  710.93  \\
           & 500   &           0.63 & \bfseries 0.59  &           22.08 & \bfseries 16.91  &           913.41 & \bfseries  710.99  \\
           & 1000  &           0.63 & \bfseries 0.58  &           22.08 & \bfseries 16.91  &           913.43 & \bfseries  711.76  \\
           & 5000  &           0.64 & \bfseries 0.59  &           22.09 & \bfseries 16.88  &           913.38 & \bfseries  711.64  \\
           & 10000 &           0.65 & \bfseries 0.59  &           22.10 & \bfseries 16.92  &           913.35 & \bfseries  711.00  \\ [0.5em]
\data{trajectory} & 1     &           2.37 & \bfseries 2.03  &           68.93 & \bfseries 54.81  &          1765.73 & \bfseries 1114.29  \\
           & 50    &  \bfseries 0.74  & \bfseries 0.74  &           32.32 & \bfseries 26.69  &           721.16 & \bfseries  529.22  \\
           & 100   &  \bfseries 0.74  &          0.75 &           10.96 & \bfseries 10.78  &           747.45 & \bfseries  514.02  \\
           & 500   &  \bfseries 0.75  & \bfseries 0.75  &            6.52 & \bfseries  6.58  &           120.86 & \bfseries   89.50  \\
           & 1000  &  \bfseries 0.74  & \bfseries 0.74  &            7.17 & \bfseries  6.64  &            79.48 & \bfseries   58.08  \\
           & 5000  &           0.75 & \bfseries 0.74  &            6.25 & \bfseries  6.18  &            76.44 & \bfseries   60.82  \\
           & 10000 &           0.74 & \bfseries 0.73  &            6.41 & \bfseries  6.28  &            74.66 & \bfseries   58.25  \\ [0.5em]
\data{web8}       & 1     &           5.11 & \bfseries 4.91  &           52.33 & \bfseries 46.86  &  \bfseries 356.63  &           411.22 \\
           & 50    &  \bfseries 3.53  &          3.54 &  \bfseries 22.78  &          23.27 &           228.59 &  \bfseries 207.61  \\
           & 100   &  \bfseries 3.37  &          3.46 &  \bfseries  6.22  &           6.87 &           143.66 &  \bfseries 138.97  \\
           & 500   &  \bfseries 3.06  &          3.16 &  \bfseries  3.93  &           4.36 &   \bfseries 17.98  &            20.08 \\
           & 1000  &  \bfseries 3.01  &          3.09 &  \bfseries  3.94  &           4.36 &   \bfseries 18.28  &            19.52 \\
           & 5000  &  \bfseries 3.00  &          3.09 &  \bfseries  3.94  &           4.37 &   \bfseries 18.20  &            19.58 \\
           & 10000 &  \bfseries 3.00  &          3.10 &  \bfseries  3.94  &           4.35 &   \bfseries 17.99  &            19.27 \\
\bottomrule
\end{tabular}

\caption{Execution time for the \data{adult8}, \data{web8}, \data{cpusmall} and \data{trajectory} datasets and different cache sizes.}
\label{tab:cache}
\end{table}

We illustrate these effects over the \data{adult8}, \data{web8}, \data{cpusmall} and \data{trajectory} datasets,
using LIBSVM's default values of $\epsilon_{\text{KKT}}=0.001$ for the termination criterion tolerance and of $\gamma=1/d$ for the kernel width; features were individually scaled to $[-1, 1]$ range.
Then, for every dataset we measure execution times for $C=10, 100, \num{10000}$ and seven different cache sizes, 
\SIlist[list-units = single]{1;50;100;500;1000;5000;10000}{\mega\byte}.
Every $C$ and cache size combination was run $10$ times and we repeated this procedure $3$ times. At the end we computed the average of the $10$ runs and for every value took the minimum of the $3$ repetitions. 

\begin{figure}[thbp]
\centering
\includegraphics[width=\textwidth]{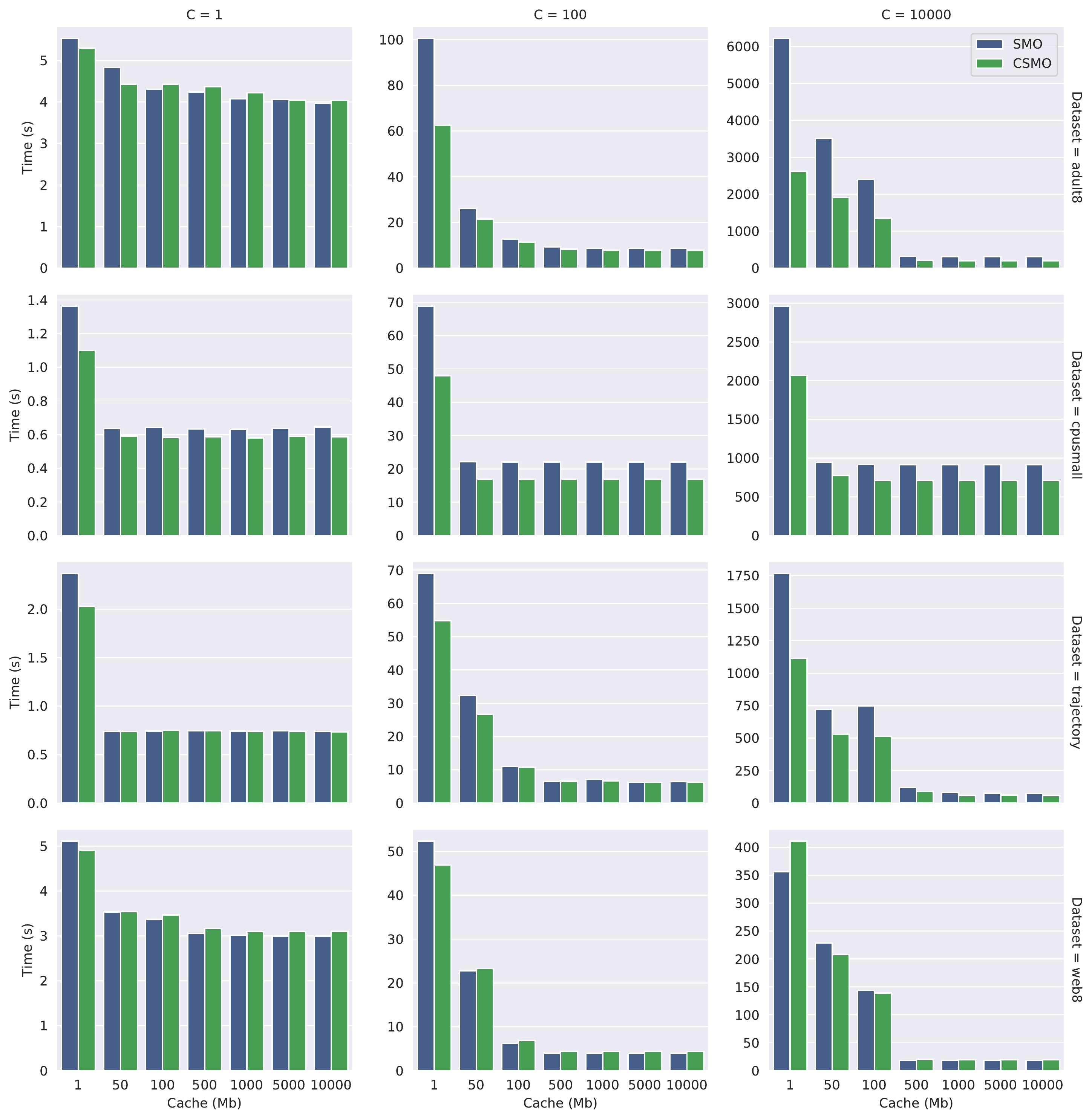}
\caption{Execution time as a function of the cache size comparison between SMO and CSMO for the \data{adult8}, \data{web8}, \data{trajectory} and \data{cpusmall} datasets and different $C$ values.}
\label{img:cache_time}
\end{figure}

Execution times in seconds together with the relative difference is given in \cref{tab:cache} for every $C$ and cache size values and all datasets. Comparing CSMO to standard SMO we can see that it is usually faster for smaller caches and larger $C$ values.
The same results are also depicted in \cref{img:cache_time}. 
The conclusions from both table and figure essentially agree with  our previous discussion.
First, as it was expected, execution times decrease as cache sizes increase, since more kernel rows can be stored and reused with the larger caches. 
Second, conjugate SMO performs better when caches are small, particularly in the longer training times arising with the stronger $C=100$ regularization.
However, there is a critical value for the cache size from which larger values offer no improvement. This upper limit for the cache size depends on the final number of support vectors which, in turn, depends on the $C$ and $\gamma$ values and, of course, sample size.
This critical value makes possible to store all the support vectors and lies between \SIlist[list-units = single]{100;500}{\mega\byte} for all our datasets (the LIBSVM default cache size is \SI{100}{\mega\byte}).

In any case, while RAM memory above \SI{100}{\giga\byte} is not uncommon on rack blade servers, it may be very well the case that several SVMs have to be trained in parallel.
A common such case is crossvalidation-based hyper-parameter searches, where a number of models well above 100 may have to be tested.
In such a situation working with, say, \SI{1}{\giga\byte} caches may not be feasible.

\subsection{Time Comparisons for Different Hyper-parameter Configurations}
\label{time_comp}

Next we will perform time comparisons for standard and conjugate SMO over a wide variety of $C$, $\gamma$ and $\epsilon$  values, again over the problems in \cref{tab:datasets}.
In fact, hyper-parameter search is the costliest task when setting up an SVM model.
We will perform our timing experiments on one such grid search scenario, using the hyper-parameter values considered in \cite{fan2005working}.
For classification they work with equi-logartihmically spaced $C$ and $\gamma$ values.
More precisely, $C$ ranges from  $2^{-5}$ to $2^{15}$ with increments of $2$ in the $\log_2$ scale, that is, $2^{-5}, 2^{-3}, \dots, 2^{13}, 2^{15}$. 
Similarly $\gamma$ ranges from $\gamma=2^{-15}$ to $\gamma=2^3$, again with $\log_2$-scaled steps of $2$. 
This gives a total of $11\times 10=110$ $(C, \gamma)$ pairs.

For regression the $\gamma$ ranges are retained, the $C$ range is reduced to $[2^{-1}, 2^{-15}]$ ($\log_2$-scaled steps of $2$) and the $\epsilon$ hyper-parameter ranges from $2^{-8}$ to $2^{-1}$, now with log-scale increments of $1$ \citep{fan2005working}.
This would result in a large $9 \times 10 \times 8=720$ number of models to be considered. 
To reduce this, in our regression experiments we will skip the top $C=2^{13},\ 2^{15}$ and bottom $\gamma = 2^{-13},\ 2^{-15}$ values of the $C$ and $\gamma$ ranges. 
The final number of grid points is hence $7 \times 8 \times 8 = 448$, two-thirds of the original size. 
Note that we are performing 5-fold cross validation, and therefore the total number of models to be fitted by either standard or conjugate SMO is $5$ times the number of grid points, i.e., $550$ for classification and \num{2240} for regression.
Finally, and according to our previous discussion, we are going to use two values for the size of the cache, a somewhat small \SI{100}{\mega\byte}, where conjugate SMO should have an advantage, and a larger \SI{1}{\giga\byte}, where both methods would be in a more balanced footing. 
Recall that, as it can be seen in \cref{img:cache_time}, there is almost no time improvement when having a cache size larger than \SI{1}{\giga\byte}. 
The convergence tolerance will by $\epsilon_{\text{KKT}} = 0.001$ in all cases. Finally, for each dataset, cache size and hyper-parameter configuration we compute the relative time difference between SMO and Conjugate SMO,
$$\text{RTD} = \frac{\text{Time}(\text{SMO}) - \text{Time}(\text{CSMO})}{\text{Time}(\text{SMO})} \times 100.$$

\begin{table}[thbp]
\centering
\footnotesize
\begin{tabular}{lS[table-format=2]S[table-format=+2]S[table-format=2.2]S[table-format=2.2]S[table-format=4]S[table-format=3.2]S[table-format=3.2]S[table-format=+1.2]}
\toprule
          & \multicolumn{2}{c}{Hyp. ($\log_2$)}  & \multicolumn{2}{c}{Accuracy (\%)} & & \multicolumn{2}{c}{Time (h)} & \\
\cmidrule(lr){2-3}\cmidrule(lr){4-5}\cmidrule(lr){7-8}
{Dataset} & {$C^\opt$} & {$\gamma^\opt$} & {SMO} & {CSMO} & {Cache} & {SMO} & {CSMO} & {RTD}\\
\midrule
\data{adult8}   & 11 &  -9 & 84.438 & 84.438 & 100  &  13.124 &    8.729 & 33.488 \\
         &    &     &        &        & 1000 &   3.317 &    2.447 & 26.232 \\ [0.5 em]
\data{cod-rna}  & 15 & -11 & 94.744 & 94.746 & 100  &  93.594 &   51.356 & 45.129 \\
         &    &     &        &        & 1000 &  61.728 &   29.911 & 51.544 \\ [0.5 em]
\data{ijcnn1}   &  5 &   1 & 98.848 & 98.848 & 100  &   2.755 &    2.602 &  5.533 \\
         &    &     &        &        & 1000 &   1.535 &    1.510 &  1.649 \\ [0.5 em]
\data{mnist1}   &  1 &  -7 & 99.812 & 99.812 & 100  & 604.167 &  543.346 & 10.067 \\
         &    &     &        &        & 1000 & 491.657 &  410.269 & 16.554 \\ [0.5 em]
\data{skin}     &  5 &  -5 & 99.972 & 99.972 & 100  &  23.873 &   23.383 &  2.051 \\
         &    &     &        &        & 1000 &  23.818 &   23.564 &  1.069 \\ [0.5 em]
\data{web8}     &  3 &  -5 & 98.864 & 98.864 & 100  &   5.757 &    5.831 & -1.288 \\
         &    &     &        &        & 1000 &   4.702 &    4.768 & -1.400 \\
\bottomrule
\end{tabular}
\caption{Results of a full hyper-parameter search (classification).}
\label{tab:svc_time}
\end{table}

\begin{table}[thbp]
\centering
\footnotesize
\begin{tabular}{lS[table-format=1]S[table-format=+1]S[table-format=+1]S[table-format=4.3]S[table-format=4.3]S[table-format=4]S[table-format=4.2]S[table-format=4.2]S[table-format=2.2]}
  \toprule
            & \multicolumn{3}{c}{Hyp. ($\log_2$)} & \multicolumn{2}{c}{MSE} & & \multicolumn{2}{c}{Time (h)} &\\
  \cmidrule(lr){2-4}\cmidrule(lr){5-6}\cmidrule(lr){8-9}
  {Dataset} & {$C^\opt$} & {$\gamma^\opt$} & {$\epsilon^\opt$} & {SMO} & {CSMO} & {Cache} & {SMO} & {CSMO} & {RTD} \\
  \midrule
\data{abalone}    & 5 & -1 & -1 &    4.500 &    4.500 & 100  &    1.877 &    1.219 & 35.052 \\
           &   &    &    &          &          & 1000 &    1.890 &    1.229 & 34.980 \\ [0.5 em]
\data{cadata}     & 9 &  3 & -1 & 3066.524 & 3066.526 & 100  &   11.495 &    9.250 & 19.525 \\
           &   &    &    &          &          & 1000 &   10.546 &    8.712 & 17.389 \\ [0.5 em]
\data{cpusmall}   & 9 & -7 & -4 &    0.028 &    0.028 & 100  &   13.379 &    9.703 & 27.472 \\
           &   &    &    &          &          & 1000 &   14.168 &   10.561 & 25.461 \\ [0.5 em]
\data{ctscan}     & 3 & -9 & -8 &    0.001 &    0.001 & 100  & 2463.891 & 2061.566 & 16.329 \\
           &   &    &    &          &          & 1000 & 1395.560 & 1220.699 & 12.530 \\ [0.5 em]
\data{trajectory} & 1 & -9 & -3 &    0.999 &    0.999 & 100  &  177.447 &  174.635 &  1.585 \\
           &   &    &    &          &          & 1000 &   24.995 &   23.898 &  4.390 \\ [0.5 em]
\data{year}       & 3 & -5 & -4 &    0.587 &    0.587 & 100  & 6546.589 & 4637.676 & 29.159 \\
           &   &    &    &          &          & 1000 & 1112.437 &  894.214 & 19.617 \\
\bottomrule
\end{tabular}
\caption{Results of a full hyper-parameter search (regression).}
\label{tab:svr_time}
\end{table}

We first report our experimental results in \cref{tab:svc_time,tab:svr_time}.
Their left columns show the optimal hyper-parameters for each classification and regression problem as well as the accuracy (as a percentage) or mean absolute error of the optimal classification and regression models.
As it can be seen, both SMO models arrive at the same optimal hyper-parameter combination and obtain the same accuracy or error 
(this is also the case in all our other hyper-parameter settings).
The tables also show at their right total accumulated times of the grid searches performed, as well as the relative time difference (RTD). 
Recall that we have applied $5$-fold cross validation;
accordingly, the times reported correspond to the 5-fold averages of the accumulated hyper-parameter search times over each fold. 

We can see that, in classification, CSMO accumulated times are always smaller than those of SMO (i.e., RTD values are positive) except for the \data{w8a} dataset, where standard SMO is slightly faster. 
Conjugate SMO times are clearly better for \data{adult8} and \data{cod-rna} and also for \data{ijcnn1}, \data{mnist1} and \data{skin}, although with a smaller edge;
CSMO is slightly behind SMO for \data{w8a}.
Also, and as expected, the time differences are usually larger for the \SI{100}{\mega\byte} cache.
The situation for regression is fairly similar with now CSMO accumulated times being smaller for all datasets.
Here it has a clear advantage for \data{abalone}, \data{cpusmall} and \data{year} and, though slightly smaller, also for \data{cadata} and \data{ctscan};
times are closer for \data{trajectory}.
Again, time differences are larger for the \SI{100}{\mega\byte} cache.


\begin{figure}[thbp]
\centering
\includegraphics[width=\textwidth]{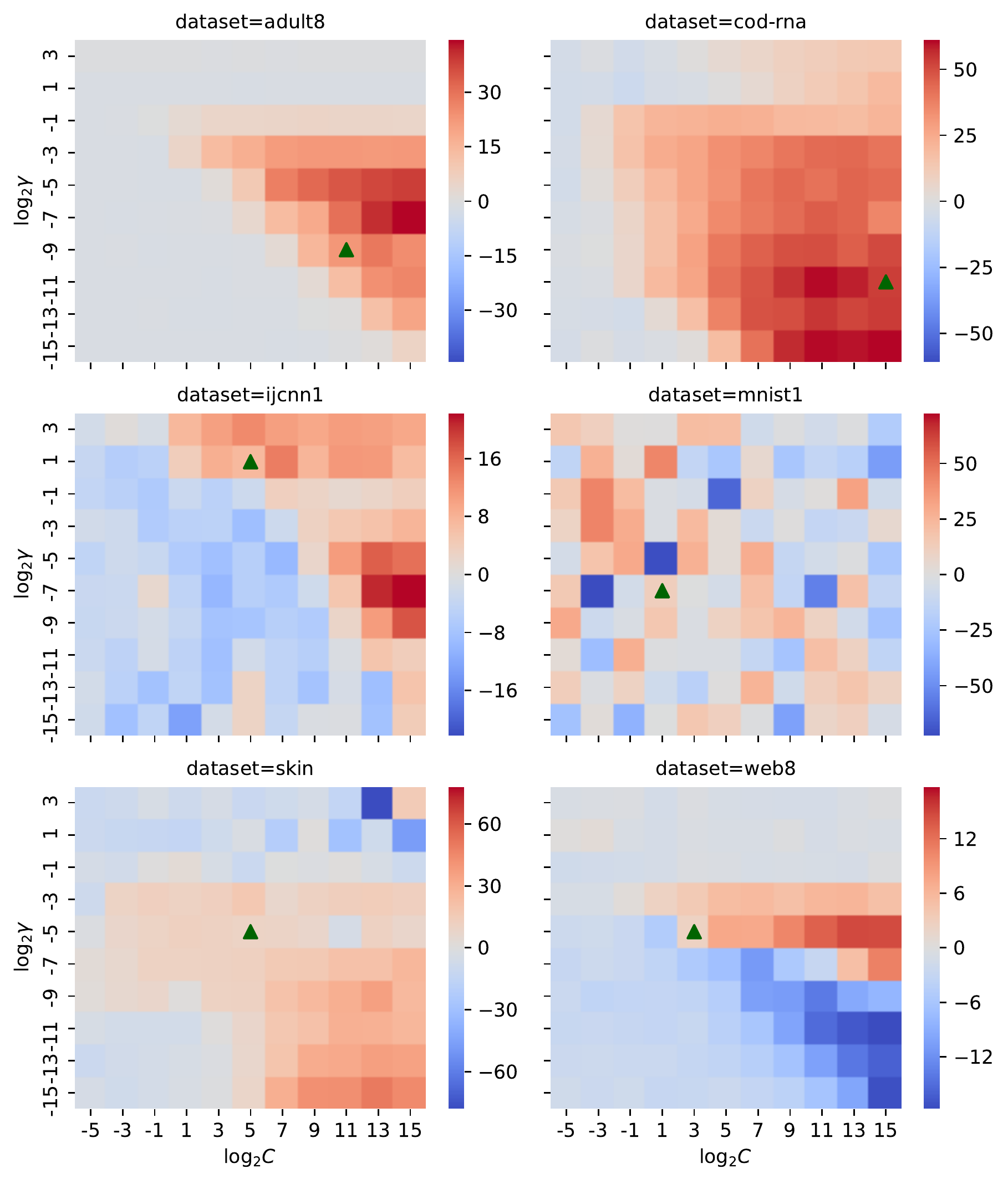}
\caption{Relative time difference heatmap with a cache size of \SI{100}{\mega\byte} for the different $C$ and $\gamma$ values of a full hyper-parameter search (classification).}
\label{img:svc_time_100}
\end{figure}

\begin{figure}[thbp]
\centering
\includegraphics[width=\textwidth]{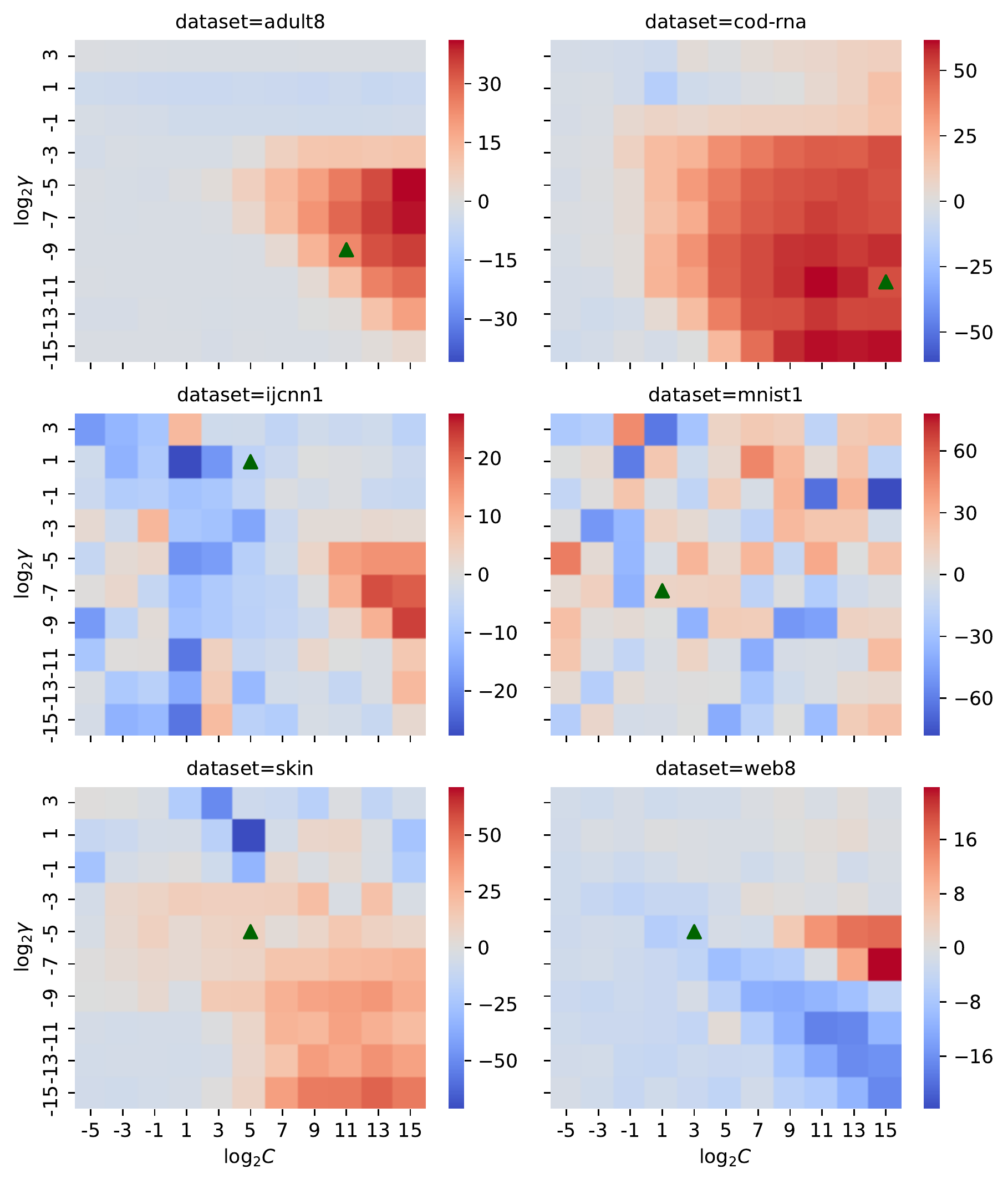}
\caption{Relative time difference heatmap with a cache size of \SI{1}{\giga\byte} for the different $C$ and $\gamma$ values of a full hyper-parameter search (classification).}
\label{img:svc_time_1000}
\end{figure}

\begin{figure}[thbp]
\centering
\includegraphics[width=\textwidth]{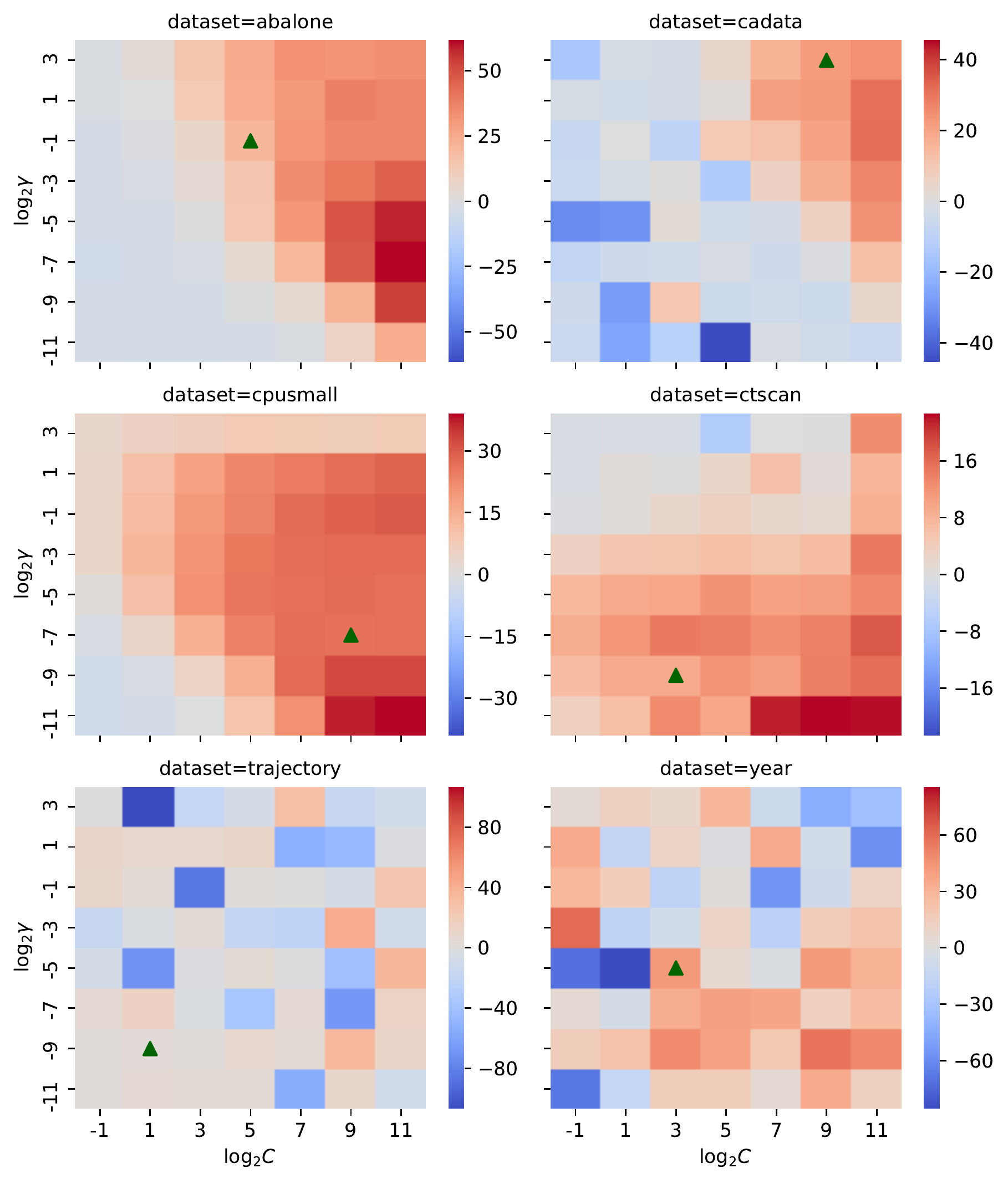}
\caption{Relative time difference heatmap with a cache size of \SI{100}{\mega\byte} for the different $C$ and $\gamma$ values and optimal $\epsilon$ of a full hyper-parameter search (regression).}\label{img:svr_time_100}
\end{figure}

\begin{figure}[thbp]
\centering
\includegraphics[width=\textwidth]{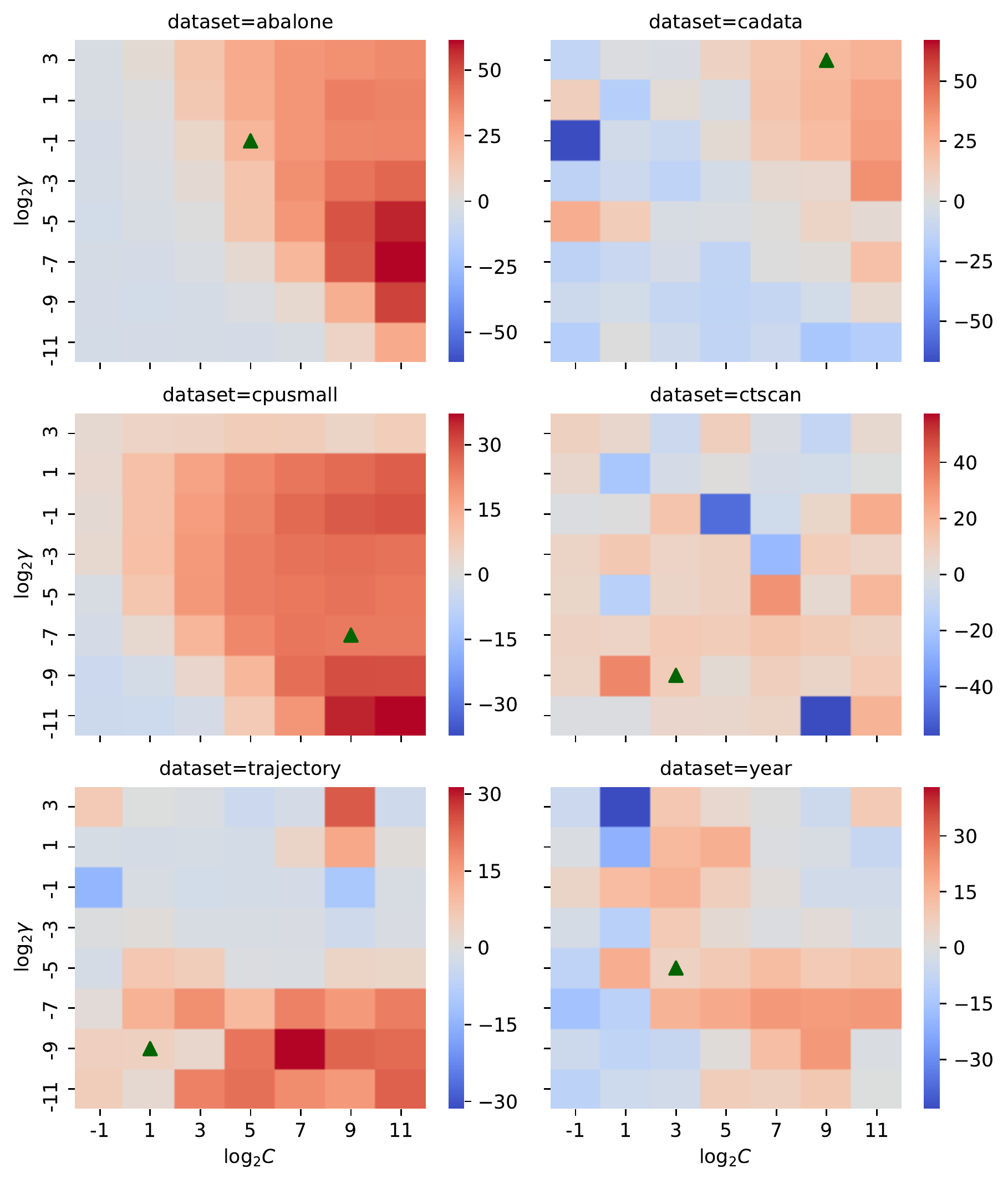}
\caption{Relative time difference heatmap with a cache size of \SI{1}{\giga\byte} for the different $C$ and $\gamma$ values and optimal $\epsilon$ of a full hyper-parameter search (regression).}\label{img:svr_time_1000}
\end{figure}

We disaggregate the relative timing differences in classification over all the hyper-parameter combinations in \cref{img:svc_time_100} for the \SI{100}{\mega\byte} cache. 
Here we can see how there is usually an upper-triangle in the $(C, \gamma)$ grid with red-colored RTD values where conjugate SMO consistently outperforms standard SMO. 
It is also interesting to see that, in general, smaller $\gamma$ values (i.e., broader Gaussians) benefit the conjugate implementation but only up to a certain threshold, from which SMO starts to be better. 
On the other hand it is quite clear that CSMO outperforms SMO for larger $C$ values (i.e., more regularization).
We point out that the optimal parameter combinations, shown with a green triangle, are often located in regions where CSMO outperforms SMO. 
This implies, first, that conjugate SMO will be effective exploring these regions when searching for optimal hyper-parameters (and more so if these searches have to be refined) and, second, that it will also help training optimal models. 
\cref{img:svc_time_1000} shows a similar situation for the $(C, \gamma)$ disaggregation now using the \SI{1}{\giga\byte} cache, although with a smaller advantage for conjugate SMO.

The $(C, \gamma)$ disaggregation for the regression problems is shown in \cref{img:svr_time_100} and \cref{img:svr_time_1000} for \SI{100}{\mega\byte} and \SI{1}{\giga\byte} caches respectively.
Given that here we are working with three-dimensional hyper-parameter grids, we report for each problem the $(C, \gamma)$ disaggregation of the times measured working with the optimal $\epsilon$ hyper-parameter.
The results here are similar to the ones for classification.
We can see that conjugate SMO timings are smaller for most of the hyper-parameter combinations in \data{abalone, cpusmall} and \data{ctscan} and also for \data{trajectory} and the \SI{1}{\giga\byte} cache.
Times for \data{cadata} and \data{year} are more even.

Finally, we point out that while in some problems the hyper-parameter disaggregation shows a structure favorable to Conjugate SMO, there are also other problems where such structure is not clear. 
But even in these cases, we see that CSMO is competitive in almost all hyper-parameter combinations. Finally,
\cref{tab:svc_time,tab:svr_time} allow us to conclude that in most problems the ``winning'' CSMO hyper-parameter combinations clearly out-weight the losing ones.


\section{Discussion}
\label{sec:disc}
In this work we have proposed a conjugate variant of the SMO algorithm, the state-of-the-art approach to solve the optimization problem required for training Support Vector Machines.
As the original SMO, the new Conjugate SMO, CSMO, can be used both for classification and regression tasks, and at each iteration it only implies a slight increase on the computational complexity compared with that of standard SMO.
In practice, however, most of the complexity of both algorithms lies in the computation of the kernel columns, since they account for around \SIrange{75}{80}{\percent} of the running time. Some computational tricks such as the cache greatly help in making the training efficient for datasets of up to 200K observations. Conjugate SMO further improves on this by reducing the number of iterations needed for convergence, and thus it also reduces the possibility of having a cache miss.

In addition, we have provided a theoretical proof of the convergence of a subsequence of the CSMO dual multipliers and of the entire primal vector sequence to a dual optimum and the unique primal one, respectively.
We have also proved a linear convergence rate when the kernel matrix is positive definite and a non-degeneracy property holds.
These conditions are also assumed for linear convergence of standard SMO, but our proofs follow different arguments and may have an interest of their own.

We have implemented Conjugate SMO within the LIBSVM library and have performed extensive experiments over 12 classification and regression datasets, most of them with a large number of samples and/or features.
CSMO often outperformed SMO over a wide range of hyper-parameter values and, moreover, the total time of performing a grid search with 5-fold cross-validation was lower for CSMO on all datsets but one, where the difference was rather small. 
As a conclusion, CSMO appears to be always competitive and often the best option for optimal hyper-parameter search.
Also, once the optimal hyperparameters are found, CSMO achieves most of the time a reduction in the training time of the resulting optimal model.

\section*{Acknowledgments}
With partial support from Spain's grant TIN2016-76406-P. 
Work supported also by project FACIL--Ayudas Fundaci\'{o}n BBVA a Equipos de Investigación Cient\'{i}fica 2016,  the UAM--ADIC Chair for Data Science and Machine Learning and the Instituto de Ingenier\'{i}a del Conocimiento.  
We gratefully acknowledge the use of the facilities of Centro de Computaci\'{o}n Cient\'{i}fica (CCC) at UAM.

\bibliographystyle{spbasic}

\begin{thebibliography}{22}
\providecommand{\natexlab}[1]{#1}
\providecommand{\url}[1]{{#1}}
\providecommand{\urlprefix}{URL }
\expandafter\ifx\csname urlstyle\endcsname\relax
  \providecommand{\doi}[1]{DOI~\discretionary{}{}{}#1}\else
  \providecommand{\doi}{DOI~\discretionary{}{}{}\begingroup
  \urlstyle{rm}\Url}\fi
\providecommand{\eprint}[2][]{\url{#2}}

\bibitem[{Chen et~al.(2006)Chen, Fan, and Lin}]{chen2006study}
Chen PH, Fan RE, Lin CJ (2006) A study on smo-type decomposition methods for
  support vector machines. IEEE Trans Neural Networks 17(4):893--908

\bibitem[{Cortes and Vapnik(1995)}]{cortes1995support}
Cortes C, Vapnik V (1995) Support-vector networks. Machine learning
  20(3):273--297

\bibitem[{Fan et~al.(2005)Fan, Chen, and Lin}]{fan2005working}
Fan RE, Chen PH, Lin CJ (2005) Working set selection using second order
  information for training support vector machines. Journal of machine learning
  research 6(Dec):1889--1918

\bibitem[{Kivinen et~al.(2004)Kivinen, Smola, and Williamson}]{Kivinen}
Kivinen J, Smola AJ, Williamson RC (2004) Online learning with kernels. {IEEE}
  Trans Signal Processing 52(8):2165--2176

\bibitem[{Lan et~al.(2019)Lan, Wang, Zhe, Cheng, Wang, and Zhang}]{Lan}
Lan L, Wang Z, Zhe S, Cheng W, Wang J, Zhang K (2019) Scaling up kernel {SVM}
  on limited resources: {A} low-rank linearization approach. {IEEE} Trans
  Neural Netw Learning Syst 30(2):369--378

\bibitem[{Lin(2002)}]{Lin_smo_conv_wout_assump}
Lin C (2002) Asymptotic convergence of an {SMO} algorithm without any
  assumptions. {IEEE} Trans Neural Networks 13(1):248--250

\bibitem[{List and Simon(2007)}]{list2007general}
List N, Simon HU (2007) General polynomial time decomposition algorithms.
  Journal of Machine Learning Research 8(Feb):303--321

\bibitem[{L{\'{o}}pez and Dorronsoro(2012)}]{LopezDorronConvSMO}
L{\'{o}}pez J, Dorronsoro JR (2012) Simple proof of convergence of the {SMO}
  algorithm for different {SVM} variants. {IEEE} Trans Neural Netw Learning
  Syst 23(7):1142--1147, \doi{10.1109/TNNLS.2012.2195198}

\bibitem[{L{\'o}pez and Dorronsoro(2015)}]{lopez2015linear}
L{\'o}pez J, Dorronsoro JR (2015) Linear convergence rate for the mdm algorithm
  for the nearest point problem. Pattern Recognition 48(4):1510--1522

\bibitem[{Ma and Belkin(2018)}]{Ma}
Ma S, Belkin M (2018) Learning kernels that adapt to {GPU}. CoRR
  abs/1806.06144, \urlprefix\url{http://arxiv.org/abs/1806.06144}

\bibitem[{Nesterov(2004)}]{NesterovIntrodLect}
Nesterov Y (2004) Introductory lectures on convex optimization : a basic
  course. Applied optimization, Kluwer Academic Publ., Boston, Dordrecht,
  London

\bibitem[{Qaadan et~al.(2019)Qaadan, Sch{\"{u}}ler, and Glasmachers}]{Qaadan}
Qaadan S, Sch{\"{u}}ler M, Glasmachers T (2019) Dual {SVM} training on a
  budget. In: Proceedings of the 8th International Conference on Pattern
  Recognition Applications and Methods, {ICPRAM} 2019, Prague, Czech Republic,
  February 19-21, 2019., pp 94--106

\bibitem[{Rahimi and Recht(2007)}]{Rahimi}
Rahimi A, Recht B (2007) Random features for large-scale kernel machines. In:
  Advances in Neural Information Processing Systems 20, Proceedings of the
  Twenty-First Annual Conference on Neural Information Processing Systems,
  Vancouver, British Columbia, Canada, December 3-6, 2007, pp 1177--1184

\bibitem[{{Rudin} and {Carlson}(2019)}]{Rudin_secrets}
{Rudin} C, {Carlson} D (2019) {The Secrets of Machine Learning: Ten Things You
  Wish You Had Known Earlier to be More Effective at Data Analysis}. arXiv
  e-prints arXiv:1906.01998

\bibitem[{Schlag et~al.(2019)Schlag, Schmitt, and Schulz}]{Schlag}
Schlag S, Schmitt M, Schulz C (2019) Faster support vector machines. In:
  Proceedings of the Twenty-First Workshop on Algorithm Engineering and
  Experiments, {ALENEX} 2019, San Diego, CA, USA, January 7-8, 2019., pp
  199--210

\bibitem[{Shalev-Shwartz et~al.(2007)Shalev-Shwartz, Singer, and
  Srebro}]{shalev2007pegasos}
Shalev-Shwartz S, Singer Y, Srebro N (2007) Pegasos: Primal estimated
  sub-gradient solver for svm. In: Proceedings of the 24th international
  conference on Machine learning, ACM, pp 807--814

\bibitem[{Steinwart et~al.(2011)Steinwart, Hush, and Scovel}]{Steinwart_2011}
Steinwart I, Hush DR, Scovel C (2011) Training svms without offset. J Mach
  Learn Res 12:141--202

\bibitem[{Thomann et~al.(2017)Thomann, Blaschzyk, Meister, and
  Steinwart}]{Thomann}
Thomann P, Blaschzyk I, Meister M, Steinwart I (2017) Spatial decompositions
  for large scale svms. In: Proceedings of the 20th International Conference on
  Artificial Intelligence and Statistics, {AISTATS} 2017, 20-22 April 2017,
  Fort Lauderdale, FL, {USA}, pp 1329--1337

\bibitem[{Torres-Barr{\'a}n(2017)}]{torres2017phd}
Torres-Barr{\'a}n A (2017) Acceleration methods for classic convex optimization
  problems. PhD thesis, Universidad Aut{\'o}noma de Madrid

\bibitem[{Torres-Barr{\'a}n and
  Dorronsoro(2016{\natexlab{a}})}]{torres2016conjugate}
Torres-Barr{\'a}n A, Dorronsoro JR (2016{\natexlab{a}}) Conjugate descent for
  the smo algorithm. In: International Joint Conference on Neural Networks
  (IJCNN), IEEE, pp 3817--3824

\bibitem[{Torres-Barr{\'a}n and
  Dorronsoro(2016{\natexlab{b}})}]{torres2016nesterov}
Torres-Barr{\'a}n A, Dorronsoro JR (2016{\natexlab{b}}) Nesterov acceleration
  for the smo algorithm. In: International Conference on Artificial Neural
  Networks, Springer, pp 243--250

\bibitem[{Yuan et~al.(2012)Yuan, Ho, and Lin}]{yuan2012improved}
Yuan GX, Ho CH, Lin CJ (2012) An improved glmnet for l1-regularized logistic
  regression. Journal of Machine Learning Research 13(Jun):1999--2030

\end{thebibliography}

\end{document}